\documentclass[aop,preprint]{imsart}

\usepackage{amsthm,amsmath,amssymb,graphicx,mathrsfs,natbib}
\RequirePackage[colorlinks,citecolor=blue,urlcolor=blue]{hyperref}

\startlocaldefs
\theoremstyle{definition}
\newtheorem{axiom}{Axiom}[section]
\theoremstyle{remark}

\theoremstyle{plain}

\newtheorem{theorem}{Theorem}

\newtheorem{corollary}{Corollary}

\newcommand{\paren}[1]{\left(#1\right)}

\newcommand{\bal}{\begin{align*}}
\newcommand{\eal}{\end{align*}}
\newcommand{\bpm}{\begin{pmatrix}}
\newcommand{\epm}{\end{pmatrix}}

\newcommand{\bfp}{\mathbf{P}}

\newcommand{\bfr}{\mathbf{R}}

\newcommand{\balign}{\begin{align*}}
\newcommand{\ealign}{\end{align*}}

\newcommand{\eq}[2]{\begin{equation} #2 \label{#1} \end{equation}}

\endlocaldefs

\begin{document}

\begin{frontmatter}

\title{Probability Theory without Bayes' Rule}
\runtitle{Probability Theory without Bayes' Rule}

\author{\fnms{Samuel G.} \snm{Rodriques}\corref{}\ead[label=e1]{sgr@mit.edu}\thanksref{t1,m1}}
\thankstext{t1}{The author thanks Maxim Rabinovich, Sina Tootoonian, David Turban, and Steven Lindell for helpful discussions and feedback.}
\thankstext{m1}{The author gratefully acknowledges the support of the Winston Churchill Foundation of the United States, and the Fannie and John Hertz Foundation.}
\address{Synthetic Neurobiology Group\\MIT Media Lab, 75 Amherst Street\\
Cambridge, MA, 02139\\\printead{e1}}
\affiliation{Department of Physics, Massachusetts Institute of Technology}

\runauthor{S.G. Rodriques}

\begin{abstract}
Within the Kolmogorov theory of probability, Bayes' rule allows one to perform statistical inference by relating conditional probabilities to unconditional probabilities. As we show here, however, there is a continuous set of alternative inference rules that yield the same results, and that may have computational or practical advantages for certain problems. We formulate generalized axioms for probability theory, according to which the reverse conditional probability distribution $P(B|A)$ is not specified by the forward conditional probability distribution $P(A|B)$ and the marginals $P(A)$ and $P(B)$. Thus, in order to perform statistical inference, one must specify an additional ``inference axiom,'' which relates $P(B|A)$ to $P(A|B)$, $P(A)$, and $P(B)$. We show that when Bayes' rule is chosen as the inference axiom, the axioms are equivalent to the classical Kolmogorov axioms. We then derive consistency conditions on the inference axiom, and thereby characterize the set of all possible rules for inference. The set of ``first-order'' inference axioms, defined as the set of axioms in which $P(B|A)$ depends on the first power of $P(A|B)$, is found to be a 1-simplex, with Bayes' rule at one of the extreme points. The other extreme point, the ``inversion rule,'' is studied in depth.
\end{abstract}



\end{frontmatter}

\section{Introduction}

 Let $A$ and $B$ be random variables taking values $a_1,\ldots,a_{N_A}$ and $b_1,\ldots,b_{N_B}$, respectively, and let $P(a_i,b_j)$ be the joint probability distribution for $A$ and $B$, which may be represented as a matrix $\bfp(A,B)$ of size $N_A \times N_B$. The conditional probability distribution $P(a_i|b_j)$ may likewise be represented as an $N_A \times N_B$ matrix, denoted $\bfp(A|B)$, with the $i$,$j$th component given according to the standard Kolmogorov axioms of probability~\cite{Kolmogorov33} by
\eq{CondProbDef}{P(a_i, b_j) = P(a_i | b_j) P(b_j).}
In the Kolmogorov framework, $P(a_i, b_j)$ is viewed as the measure of the intersection of sets representing $a_i$ and $b_j$. It thus follows from the symmetry of the intersection operator that $P(a_i, b_j) = P(b_j, a_i)$. From the definition of conditional probabilities we can thus derive Bayes' rule, which gives the relationship between $P(A|B)$ and $P(B|A)$:
\eq{BayesRuleFirst}{P(a_i | b_j) = \frac{P(b_j | a_i) P(a_i)}{P(b_j)}.}

However, Bayes' rule gives rise to some surprising behavior. By unitarity and additivity, it follows that
\eq{}{P(a_i) = \sum_j P(a_i, b_j),}
and hence
\eq{eq4}{P(a_i) = \sum_j P(a_i|b_j) P(b_j).}
Equation~\ref{eq4} is simply the statement that the vector $\vec{P}(A)$ with components $P(a_i)$ is obtained from the vector $\vec{P}(B)$ with components $P(b_j)$ by the action of $\bfp(A|B)$. Likewise for $B$, we have
\eq{}{P(b_j) = \sum_i P(b_j| a_i) P(a_i).}
By substitution, it follows that
\eq{InversionFallacy}{P(b_k) = \sum_{ij} P(b_k| a_i) P(a_i|b_j)P(b_j).}
Given $\bfp(A|B)$, equation~\ref{InversionFallacy} must hold for all choices of the probability distribution $\vec{P}(B)$. One natural conclusion would be that $\bfp(A|B)$ and $\bfp(B|A)$ are inverses, i.e.,
\eq{}{\sum_i P(b_k | a_i) P(a_i | b_j) = \delta_{kj}.}
However, the matrix inverse of $\bfp(A|B)$ will have negative values and values greater than 1 in general. Instead, given constant $\bfp(A|B)$, $\bfp(B|A)$ varies with $\vec{P}(B)$ according to Bayes' rule, in such a way that $\vec{P}(B)$ is always an eigenvalue of the matrix $\bfp(B|A) \bfp(A|B)$ with eigenvalue 1. This situation is particularly dissatisfying when one thinks of $\bfp(A|B)$ as describing the action of a channel, in which case $\bfp(B|A)$ would describe the result of running the channel backwards. In this case, Bayes' rule does not allow a simultaneous description of the forward channel and the reverse channel independent of the distribution that goes through it. One imagines that the matrix inverse might be a more appropriate description of the reverse channel than the matrix $\bfp(B|A)$ obtained via Bayes' rule.

The matrix inverse of $\bfp(A|B)$ shows up in other ways as well. For example, suppose we are given $\vec{P}(A)$ and $\bfp(A|B)$, and wish to infer $\vec{P}(B)$. This task is not possible using Bayes' rule, since $\bfp(B|A)$ depends on $\vec{P}(B)$; one requires a prior distribution. Nonetheless, $\vec{P}(B)$ may always be obtained from $\vec{P}(A)$ via the inverse of the matrix $\bfp(A|B)$:
\eq{InfUsingInv}{\vec{P}(B) = \bfp(A|B)^{-1} \vec{P}(A).}
Within the strict framework of Kolmogorov probability theory, one must view this procedure as a coincidence, since the matrix $P(A|B)^{-1}$ is not a conditional probability distribution.

As we show here, the application of $\bfp(A|B)^{-1}$ in this way may be explained using an alternative axiomatization of probability theory in which Bayes' rule is replaced by the requirement that
\eq{Inv}{\bfp(B|A) = \bfp(A|B)^{-1},}
referred to as the ``inversion rule.'' Moreover, the inversion rule is only one of a continuous set of possible alternative axiomatizations of probability theory, each of which comes with its own rule for statistical inference. In the second section, we provide the general axiomatic framework for probability theory that we will consider in the remainder of the paper. In the third section, we explore the specific axiom that gives rise to equation~\eqref{Inv} as a rule for inference, and show how it may be used alongside Bayes' rule in practice. Finally, in the fourth section, we characterize the set of all possible inference rules. We find that Bayes' rule and the inversion rule form a conjugate pair, and that they can be used to generate all other rules for inference.

\section{Axioms}
We seek to formulate a set of axioms that preserves the notion of a probability distribution in its entirety without Bayes' rule. Since Bayes' rule is a consequence of the symmetry of the intersection operator, it is necessary to replace Kolmogorov's axiomatization in terms of set theory by a similar axiomatization in terms of sequences. The primary hurdle will be to establish analogs of the union and intersection operators for sequences, which is necessary to facilitate comparison with the Kolmogorov axioms.

We consider the case of $N$ random variables $A^{(1)},\ldots,A^{(N)}$, where $A^{(i)}$ takes values $a^i_j \in J_i$, for $j = 1,\ldots,M_i$. Let $S$ be the set of all permutations of the integers from 1 to $N$. Given some $s \in S$, we denote by $E_s$ the set of all sequences $Q$ of $N$ elements such that the $i$th element in $Q$ is in $J_{s(i)}$. $Q$ is then said to have order $s$. Given sequences $Q$ and $Q'$, $Q$ is said to be a subsequence of $Q'$ if $Q$ may be obtained by removing elements from $Q'$. We denote by $F_s$ the set of all subsequences of elements of $E_s$. A sequence $Q$ is then said to have order $s$ if it is in $F_s$. For any sequence $Q$, we denote by $\bar{Q}$ the set of elements appearing at some position in $Q$, referred to as the ``membership set'' of $Q$. We write $Q \subseteq_s Q'$ if $Q$ and $Q'$ both have order $s$, and if $\bar{Q} \subseteq \bar{Q}'$.

For any $Q \in F_s$, we define $R_s(Q)$ to be the set of sequences in $E_s$ of which $Q$ is a subsequence, i.e.,
\eq{}{R_s(Q) \equiv \{Q' \in E_s | Q \subseteq_s Q'\}.}
We may now evidently define the operator $\cap_s$ such that for all $Q_1,Q_2 \in F_s$, 
\eq{}{Q_1 \cap_s Q_2 = \{Q' \in R_s(Q_1) \cap R_s(Q_2)\}.}
Likewise, we define the operator $\cup_s$ such that for all $Q_1$, $Q_2 \in F_s$,
\eq{}{Q_1 \cup_s Q_2 = \{Q' \in R_s(Q_1) \cup R_s(Q_2)\}.}

The axioms are then as follows, and are formulated for the greatest similarity to the original axioms of Kolmogorov:

\begin{axiom}\label{Kol1}
For all $s \in S$ and for all $Q \in E_s$, $P(Q)$ is a real-valued function referred to as the probability of $Q$.
\end{axiom}

\begin{axiom}\label{Kol2}
For all $s \in S$ and for all $Q \in F_s$,
\eq{}{P(Q) = \sum_{Q' \in R_s(Q)} P(Q').}
\end{axiom}
Crucially, we note that because $P$ is not explicitly a function of $s$, axiom~\ref{Kol2} implies that for all $s$,$s' \in S$ and for all $Q \in F_s \cap F_{s'}$,
\eq{CausalityConstraint}{\sum_{Q' \in R_s(Q)} P(Q')= P(Q) = \sum_{Q' \in R_{s'}(Q)} P(Q').}
Equation~\ref{CausalityConstraint} is equivalent to the statement that the marginal probability distribution over a subset of variables depends only on the marginal ordering of the variables in the subset.

\begin{axiom}\label{Kol3}
For all $s \in S$, we have
\eq{}{P(E_s)=1.}
\end{axiom}

\begin{axiom}\label{Kol4}
For all $Q_1,Q_2 \in F_s$ such that $R_s(Q_1) \cap R_s(Q_2) = \phi$,
\eq{}{P(Q_1 \cup_s Q_2) = P(Q_1) + P(Q_2).}
\end{axiom}

For any given value of $s$, the axioms~\ref{Kol1} through~\ref{Kol4} are identical to the standard statement of Kolmogorov's axioms without the requirement that $P(Q)$ be non-negative. Thus, the axioms are consistent, but not complete.

For the sake of concreteness, we will henceforth focus on the case of two variables, $A$ and $B$, so that there are only two orderings. Then a conditional probability may be easily defined\footnote{The conditional probability may be easily defined in the same way for any number of variables, but the notation required to treat the general case is very cumbersome, so has been avoided here.}:
\eq{}{P(a_i|b_j) \equiv \frac{P(a_i,b_j)}{P(b_j)}.}
Note that the ordering of the variables should be read off from right to left, so $P(a_i,b_j)$ is the probability of observing $a_i$ and $b_j$ relative to the ordering in which $B$ precedes $A$. This convention has been chosen to keep with the standard convention for conditional probabilities, according to which the conditioning variables are written on the right, and the conditioned variables are written on the left.

We are interested in our ability to perform statistical inference within this system of axioms. In the problem of statistical inference, we consider an observable variable $A$ and a hidden variable $B$. The goal is to determine a posterior for the hidden variable, $P(b_j | a_i)$, in terms of a model, $P(a_i | b_j)$, and a prior, $P(b_j)$. Within the current axiomatic system, the posterior distribution $P(b_j | a_i)$ is constrained in terms of $P(a_i | b_j)$ and $P(b_j)$ via equation~\ref{CausalityConstraint}, which applied to the current system becomes
\eq{CausalityConstraint2}{\sum_j P(a_i,b_j) = P(a_i) = \sum_j P(b_j, a_i).}
However, this constraint does not specify $P(b_j | a_i)$ uniquely in terms of $P(a_i | b_j)$. Thus, in order to perform statistical inference, it is necessary to impose an additional ``inference axiom'' that provides us with a rule for inferring $P(b_j | a_i)$ in terms of $P(a_i | b_j)$ and $P(b_j)$. Bayes' rule is one possible rule of inference, and may be recovered from the current axiomatization by requiring that the joint probability over some set of variables is independent of the ordering of those variables\footnote{In the general formalism, the axiom is that for all $Q \in E_s$ and $Q' \in E_{s'}$ such that $\bar{Q} = \bar{Q}'$, $P(Q) = P(Q')$.}, in which case $P(a_i,b_j) = P(b_j,a_i)$ for all $i$ and $j$, and hence
\eq{BayesRule}{P(b_j, a_i) = P(a_i | b_j) \frac{P(b_j)}{P(a_i)}.}
As we show below, however, there are many other choices of inference axioms that will nonetheless yield correct results.

In all that follows, we will assume that $\bfp(A|B)$ is invertible, which ensures that the problem of inference will be uniquely solvable. 

\section{The Inversion Rule}
Before examining the set of all possible inference axioms, it will be useful to examine one particular inference rule, the inversion rule, in depth. The inversion rule is given by equation~\eqref{Inv}, and is noteworthy for being the only inference rule in which $\bfp(B|A)$ is specified exclusively in terms of $\bfp(A|B)$. It is thus associated with the following inference axiom:

\begin{axiom}\label{IndependenceAx}
$\bfp(A|B)$ and $\bfp(B|A)$ may be specified independently of $\vec{P}(B)$ and $\vec{P}(A)$.
\end{axiom}

We prove that this axiom uniquely specifies the inversion rule as follows:

\begin{theorem}[Inversion Theorem]\label{InvThm}
For all non-product distributions $P(A,B)$, the matrices $\bfp(A|B) \bfp(B|A)$ and $\bfp(B|A) \bfp(A|B)$ are equal to the identities of size $N_A$ and $N_B$ respectively.
\end{theorem}

\begin{proof}
From the definition of conditional probability, we have
\eq{}{P(a_i) = \sum_j P(a_i | b_j) P(b_j),}
and
\eq{}{P(b_j) = \sum_k P(b_j | a_k) P(a_k).}
From eq.~\eqref{CausalityConstraint}, we have
\eq{InvThm1}{P(a_i) = \sum_{j,k} P(a_i | b_j) P(b_j | a_k) P(a_k).}
Eq.~\eqref{InvThm1} may be recast in matrix form as
\eq{InvThm2}{\vec{P}(A) = \bfp(A|B) \bfp(B|A) \vec{P}(A).}
Equation~\ref{InvThm2} must hold for all choices of $\vec{P}(A)$. However, $\bfp(A|B) \bfp(B|A)$ is independent of $\vec{P}(A)$ by axiom~\ref{IndependenceAx}. It follows that $\bfp(A|B) \bfp(B|A)$ must be the identity. The proof for $\bfp(B|A) \bfp(A|B)$ is identical.
\end{proof}
Equation~\eqref{InvThm2} may in general still be applied even if $\bfp(A,B)$ does not have full rank by restricting the inversion to the support of $\bfp(A,B)$. However, when $\bfp(A,B)$ is a product distribution, all quantities associated with all other orderings of the variables are undefined. In the case of a product distribution, it is impossible to decouple $\bfp(A|B)$ from $\vec{P}(A)$ since $P(a_i|b_j) = P(a_i)$ for all $j$, so axiom~\ref{IndependenceAx} is inconsistent with the other axioms.

The joint distributions for the two different orderings may be related as follows:

\begin{theorem}\label{ReversalThm}
Let $\bfp(A,B)$ be a non-product joint distribution. Then,
\eq{}{\bfp(B,A) = \bfp(B) \bfp(A,B)^{-1} \bfp(A).}
\end{theorem}

\begin{proof}
Combining~\eqref{InvThm2} with the definition of conditional probability, it is clear that
\eq{MainDeriv0}{\bfp(A|B) = \bfp(A,B) \bfp(B)^{-1}}
and
\eq{MainDeriv1}{\bfp(B|A) = \bfp(B,A) \bfp(A)^{-1}.}
Inverting~\eqref{MainDeriv0} and equating the left-hand sides of the two equations, we obtain
\[\bfp(B,A) \bfp(A)^{-1} = \bfp(B) \bfp(A,B)^{-1},\]
from which it follows that
\eq{MainDeriv}{\bfp(B,A) = \bfp(B) \bfp(A,B)^{-1} \bfp(A).}
\end{proof}
Clearly, if all the entries of $\bfp(A,B)$ are positive then the entries of $\bfp(B,A)$ will neither be strictly positive nor smaller than 1 in general.

A further interesting consequence of Theorem~\ref{ReversalThm} is that the function $P$ is completely specified by the joint distribution for only a single order, and hence contains the same amount of information as the joint distribution in the Kolmogorov framework. More generally, as long as the chosen inference axiom specifies $\bfp(B|A)$ completely in terms of $\bfp(A|B)$ and $\bfp(B)$, $\bfp(B,A)$ will be completely specified in terms of $\bfp(A,B)$. For this reason, we are always free to choose a preferred ordering of the variables, relative to which all joint, conditional, and marginal probabilities agree with the Kolmogorov probabilities:

\begin{corollary}\label{AgCor}
Let $P_K$ be a non-product probability distribution obeying the Kolmogorov axioms of probability. Then there exists a probability distribution $P_I$ obeying axioms~\ref{Kol1} through~\ref{Kol4} and~\ref{IndependenceAx} such that for all $i$ and $j$, $P_K(a_i,b_j) = P_I(a_i,b_j)$.
\end{corollary}

Corollary~\ref{AgCor} allows us to use the inversion rule to perform statistical inference on Kolmogorov probabilities as follows. Given a model $\bfp_K(A|B)$ and an empirical estimate $\tilde{P}_K(A)$ for $\vec{P}_K(A)$, we consider the probability distribution $\bfp_I(A,B)$ in the inversion framework, defined such that $\bfp_I(A,B) = \bfp_K(A,B)$. Then, the best empirical estimate for $\tilde{P}_I(B)$ using the inversion rule is given by
\begin{align}
\tilde{P}_I(B) &= \bfp_I(B|A) \tilde{P}_I(A)\nonumber\\
&= \bfp_I(A|B)^{-1} \tilde{P}_I(A)\nonumber\\
&= \bfp_K(A|B)^{-1} \tilde{P}_K(A)
\end{align}

Furthermore, in the limit as $\tilde{P}_K(A) \rightarrow \vec{P}_K(A)$, the estimate $\tilde{P}_I(B)$ obviously converges to $\vec{P}_I(B)$. Because $\vec{P}_I(B) = \vec{P}_K(B)$, we may interpret $\tilde{P}_I(B)$ as an approximation to $\tilde{P}_K(B)$:
\eq{KolApprox1}{\tilde{P}_K(B) = \bfp_K(A|B)^{-1} \tilde{P}_K(A).}
Equation~\ref{KolApprox1} cannot be justified on the basis of the axioms of Kolmogorov probability theory. In the Kolmogorov framework, equation~\ref{InfUsingInv} can only be used to relate the specific distributions $\vec{P}_K(A)$ and $\vec{P}_K(B)$. Because $\bfp_K(A|B)^{-1}$ is not a conditional probability distribution, there is no guarantee that $\tilde{P}_K(B)$ obtained using~\eqref{KolApprox1} is a probability distribution. Nonetheless, using the inversion framework, we have proven the validity of~\eqref{KolApprox1} as a method of performing inference.

One interpretation of the preceding discussion is that it is possible to use the inversion rule to infer a prior from the data. Surprisingly, because the inversion rule and Bayes' rule are independent, one may then apply Bayes' rule to do Bayesian inference using the inferred prior. We have

\begin{align}
\tilde{\bfp}_K(B|A) &= \tilde{\bfp}(B) \bfp_K(A|B)^T \tilde{\bfp}(A)^{-1}\nonumber\\
&= \text{Diag}\left[\bfp_K(A|B)^{-1} \tilde{P}(A)\right]\bfp_K(A|B)^T \tilde{\bfp}(A)^{-1}
\label{InfEq}
\end{align}
Here, $\text{Diag}\left[\vec{v}\right]$ is the diagonal matrix with the entries of $\vec{v}$ along the diagonal. The estimate $\tilde{\bfp}_K(B|A)$ converges to $\bfp_K(B|A)$ as $\tilde{P}(A) \rightarrow \vec{P}(A)$, so $\tilde{P}(B)$ is guaranteed to be a ``good'' prior, even if it is not a Kolmogorov probability distribution.

\section{All Possible Inference Axioms}
We now examine the set of all possible inference axioms that can be added to axioms~\ref{Kol1} through~\ref{Kol4}. An inference axiom is any way of specifying $\bfp(B|A)$ in terms of $\bfp(A|B)$ and $\vec{P}(B)$. For an inference axiom to be consistent with axioms~\ref{Kol1} through~\ref{Kol4}, the matrix $\bfp(B|A)$ must satisfy equation~\ref{CausalityConstraint2}. For equation~\ref{CausalityConstraint2} to be satisfied, the columns of $\bfp(B|A)$ must sum to 1, and the marginal distributions must not depend on the order of the joint distribution over which one chooses to marginalize. From this latter requirement and the definition of conditional probabilities, we derive the requirement that

\eq{AxiomReq}{\vec{P}(B) = \bfp(B|A) \bfp(A|B) \vec{P}(B).}
Multiplying both sides by $\bfp(B|A)^{-1}$, we obtain

\eq{}{\bfp(B|A)^{-1} \vec{P}(B) = \bfp(A|B) \vec{P}(B).}
Because the right-hand side is equal to $\vec{P}(A)$, we can multiply both sides by $\bfp(A)^{-1}$, obtaining
\eq{Axiomeq}{\bfp(A)^{-1} \bfp(B|A)^{-1} \vec{P}(B) = \vec{1}_A,}
where $\vec{1}_A$ is the vector of $1$s with support equal to the support of $\vec{P}(A)$. We define $\bfr$ by
\eq{Rdef}{\bfp(B|A)^{-1} =\bfp(A) \bfr\; \bfp(B)^{-1}.}
Substituting equation~\eqref{Rdef} into~\eqref{Axiomeq}, we obtain
\eq{REigEq}{\bfr \vec{1}_B = \vec{1}_A.}
For any matrix $\bfr$ obeying~\eqref{REigEq} for all choices of $\bfp(B)$ and $\bfp(A|B)$, there is definition of $\bfp(B|A)$ that satisfies~\eqref{CausalityConstraint2}, given by
\eq{QEqn}{\bfp(B|A) = \bfp(B) \bfr^{-1} \bfp(A)^{-1}.}
In retrospect, equation~\ref{QEqn} is obvious. For all choices of $\bfp(A)$ and $\bfp(B)$, $\bfp(B|A)$ must map $\vec{P}(A)$ onto $\vec{P}(B)$. This requirement is clearly satisfied only if the matrix $\bfr$ given in~\eqref{QEqn} maps $\vec{1}_B$ onto $\vec{1}_A$. For the columns of $\bfp(B|A)$ to sum to 1, we further require\footnote{One example of a definition for $\bfp(B|A)$ that obeys~\eqref{CausalityConstraint2} but has columns that do not sum to 1 can be obtained in the case of $N_A = N_B$ by setting $\bfr$ equal to $\bfp(A|B)^T$.}
\eq{SecondRequirement}{\bfp(B|A)^T \vec{1}_B = \vec{1}_A.}
From~\eqref{QEqn}, this is equivalent to the statement that
\[\vec{P}(A) = (\bfr^T)^{-1}\vec{P}(B),\]
or equivalently,
\eq{REigEq2}{\bfr^T \vec{P}(A) = \vec{P}(B),}
so $\bfr$ must itself be the transpose of a conditional probability distribution. The challenge is now to enumerate the set of matrices $\bfr$ that satisfy both~\eqref{REigEq} and~\eqref{REigEq2}. We can categorize the choices by the order of their dependency on $\bfp(A|B)$.

\subsection{Zeroth-order axioms}
It is possible to specify $\bfp(B|A)$ entirely independently of $\bfp(A|B)$, while still satisfying both of the requirements of~\eqref{CausalityConstraint2}. One such specification involves setting all of the columns of $\bfp(B|A)$ to be equal to $\vec{P}(B)$, or equivalently, setting $\bfr^{-1}$ to be the matrix in which all of the rows are given by $\vec{P}(A)$. However, this axiom is uninteresting from the perspective of inference, since it is impossible to infer $\vec{P}(B)$ without knowing $\vec{P}(B)$ to begin with. The corresponding rule says, if $\vec{P}(B)$ is already known, one can simply ignore all measurements.

\subsection{First-order axioms}

The set of first-order axioms is a one-dimensional convex hull, the two extreme points of which form a conjugate pair. The first extreme point is obtained by setting $\bfr$ equal to $(P(A|B)^T)^{-1}$. Applying~\eqref{QEqn}, we recover Bayes' rule:
\eq{BayesRuleAgain}{\bfp(B|A) = \bfp(B) \bfp(A|B)^T \bfp(A)^{-1}.}
The matrix $\bfp(B|A)^T$ given in~\eqref{BayesRuleAgain} also maps $\vec{1}_B$ onto $\vec{1}_A$, and also satisfies~\eqref{REigEq2}, so it is also a candidate for $\bfr$. This results in the second axiom:
\eq{}{\bfr = \bfp(A)^{-1} \bfp(A|B) \bfp(B)}
Applying~\eqref{QEqn}, we find the inversion rule,
\eq{InvRuleAgain}{\bfp(B|A) = \bfp(A|B)^{-1}.}
The conjugacy of Bayes' rule and the inversion rule is now obvious. Denoting by $\bfr_K$ the axiom specifying Bayes' rule and by $\bfr_I$ the axiom specifying the inversion rule, we have
\eq{}{\bfp(B|A)_K^T = \bfr_I,}
and likewise,
\eq{}{\bfp(B|A)_I^T = \bfr_K.}
The set of all first-order axioms may then be obtained by noting that any convex combination of first-order axioms is also a first-order axiom, so the set of first-order axioms is a convex hull. Since there are no other matrix functions first order in $P(A|B)$ that obey~\eqref{REigEq} and~\eqref{REigEq2}, the entire set of first-order axioms may be parametrized by
\eq{}{\bfr = p \bfp(A)^{-1} \bfp(A|B)\bfp(B) + (1-p) (\bfp(A|B)^T)^{-1},}
for $p \in [0,1]$.

\subsection{Higher-order Axioms}
Just as any convex combination of $\bfr$ matrices is also a $\bfr$ matrix, any product of $\bfr$ matrices and $\bfr$-matrix inverses is a $\bfr$ matrix as well. To see this, note that if $\bfr$ maps $\vec{1}_B$ to $\vec{1}_A$, then $\bfr^{-1}$ maps $\vec{1}_A$ to $\vec{1}_B$. Likewise, if $\bfr^T$ maps $\vec{P}(A)$ to $\vec{P}(B)$, then $(\bfr^T)^{-1}$ maps $\vec{P}(B)$ to $\vec{P}(A)$. The set of higher-order axioms may thus be generated by combining different first-order $\bfr$ matrices in the pattern 
\eq{}{\bfr = \bfr_1 \bfr_2^{-1} \bfr_3\ldots,}
corresponding to the rule
\eq{}{\bfp(B|A) = \bfp(B|A)_1 \bfp(B|A)_2^{-1} \bfp(B|A)_3,}
where $\bfp(B|A)_i$ is the conditional probability associated with $\bfr_i$. Evidently, rules only exist at odd orders, and clearly the set of $n$th-order rules is the convex hull of the $n$th-order product set of the set of first-order rules. As an example of a rule at higher-order, we have

\begin{align}
\bfr&= \bfr_K \bfr_I^{-1} \bfr_K\nonumber\\
&=(\bfp(A|B)^T)^{-1} \paren{\bfp(B)^{-1} \bfp(A|B)^{-1} \bfp(A)} (\bfp(A|B)^T)^{-1},\label{ThirdOrderTheory}
\end{align}
which implies
\eq{}{\bfp(B|A) = \bfp(B) \bfp(A|B)^T \bfp(A)^{-1} \bfp(A|B) \bfp(B) \bfp(A|B)^T \bfp(A)^{-1}.}

This particular rule has the property that it discounts the evidence in favor of the prior and the model, as is evident from the fact that in this system, $\bfp(B|A)$ depends only at first order on $\bfp(A)$, while it depends on $\bfp(B)$ at second order and $\bfp(A|B)$ at third order. Although it is more costly to implement computationally, it also displays some properties that may be useful for machine learning applications. For example, unlike for Bayes' rule or the inversion rule, the columns of $\bfp(B|A)$ will not in general sum to 1 unless the correct values of $\bfp(B)$ and $\bfp(A)$ are used in the calculation, thus providing a test of convergence.

\section{Conclusion}

We have shown that the axiomatization of probability theory in terms of set theory originally formulated by Kolmogorov imposes strong conditions on the form of the reverse conditional probabilities. These conditions ensure desirable properties required for consistency with the frequentist interpretation of probability theory, including positivity and correct behavior when $P(A,B)$ is a product distribution or not invertible. However, they are not necessary for correct statistical inference. Within a more general axiomatization, one can use alternative rules for statistical inference in parallel with Bayes' rule, leading to unexpected new ways of performing inference. As a result, alternative rules of inference may prove more useful than Bayes' rule for some applications.

In addition, we ask whether there may be physical circumstances in which an alternative axiomatization of probability theory would be preferred to the Kolmogorov theory. For example, it is known that although observable probabilities are always positive, the underlying distributions describing quantum mechanical states do not necessarily obey positivity~\cite{Wigner32}. Indeed, the violation of positivity by quantum states has been shown to be a fundamental distinguishing property of quantum mechanics~\cite{Spekkens08}. It may be that for describing these quantum mechanical probability distributions, which are not necessarily subject to a standard frequentist interpretation, an alternative axiomatization of probability is more suitable.

\bibliographystyle{plainnat}
\bibliography{ProbabilityCitations}






\end{document}